\documentclass[11pt]{article}
\usepackage{tikz}
\usepackage{ucs}
\usepackage{amssymb}
\usepackage{amsthm}
\usepackage{amsmath}
\usepackage{latexsym}
\usepackage[cp1251]{inputenc}
\usepackage[english]{babel}
\usepackage{graphicx}
\usepackage{wrapfig}
\usepackage{caption}
\usepackage{subcaption}
\usepackage{txfonts}
\usepackage{lmodern}
\usepackage{mathrsfs}
\usepackage{algorithm}
\usepackage{multicol}
\usepackage{enumerate}
\usepackage{titlesec}
\usepackage{paralist}
\usepackage{mathtools}
\titleformat{\paragraph}
{\normalfont\normalsize\bfseries}{\theparagraph}{1em}{}
\titlespacing*{\paragraph}
{0pt}{3.25ex plus 1ex minus .2ex}{1.5ex plus .2ex}

\newcommand*{\myproofname}{Proof}

\newcommand*\xor{\oplus}

\usepackage{wasysym}

\usepackage{fullpage}

\def\qed{\hfill\ifhmode\unskip\nobreak\fi\qquad\ifmmode\Box\else\hfill$\Box$\fi}

\usepackage{amsthm, amssymb, amsmath}
\usepackage{verbatim, float}
\usepackage{mathrsfs}
\usepackage{graphics}
\usepackage[usenames,dvipsnames]{pstricks}
\usepackage{epsfig}
\usepackage{pst-grad} 
\usepackage{pst-plot} 
\usepackage{cases}
\usepackage{algorithmic}
\usepackage{subcaption}
\usepackage{tikz,pgfplots}
\usepackage{chngcntr}

\newlength\figureheight 
\newlength\figurewidth 
\usepackage[top=1.8cm, bottom=1.8cm, left=1.8cm, right=1.8cm]{geometry}

\newtheorem{theorem}{Theorem}
\newtheorem{corollary}[theorem]{Corollary}

\newtheorem{lemma}[theorem]{Lemma}
\newtheorem{claim}[theorem]{Claim}
\newtheorem{proposition}[theorem]{Proposition}
\newtheorem{remark}[theorem]{Remark}
\theoremstyle{definition}
\newtheorem{defn}[theorem]{Definition}

\newtheorem{example}[theorem]{Example}
\theoremstyle{remark}
\title{A combinatorial proof for the secretary problem with multiple choices}

\author{
Xujun Liu\thanks{Department of Foundational Mathematics, Xi'an Jiaotong-Liverpool University, Suzhou, Jiangsu Province, 215123, China, xujun.liu@xjtlu.edu.cn} \and
Olgica Milenkovic\thanks{Department of Electrical and Computer Engineering, University of Illinois, Urbana-Champaign, Urbana, IL, 61801, USA, milenkov@illinois.edu} \and
George V. Moustakides\thanks{Department of Electrical and Computer Engineering, University of Patras, Rio, 26500, Greece, moustaki@upatras.gr}
 }

\begin{document}
\maketitle

\begin{abstract} 

The Secretary problem is a classical sequential decision-making question that can be succinctly described as follows: a set of rank-ordered applicants are interviewed sequentially for a single position. Once an applicant is interviewed, an immediate and irrevocable decision is made if the person is to be offered the job or not and only applicants observed so far can be used in the decision process. The problem of interest is to identify the stopping rule that maximizes the probability of hiring the highest-ranked applicant. A multiple choice version of the Secretary problem, known as the Dowry problem, assumes that one is given a fixed integer budget for the total number of selections allowed to choose the best applicant. It has been solved using tools from dynamic programming and optimal stopping theory. We provide the first combinatorial proof for a related new \emph{query-based model} for which we are allowed to solicit the response of an expert to determine if an applicant is optimal. Since the selection criteria differ from those of the Dowry problem we obtain nonidentical expected stopping times.

Our result indicates that an optimal strategy is the $(a_s, a_{s-1}, \ldots, a_1)$-strategy, i.e., for the $i^{th}$ selection, where $1 \le i \le s$ and $1 \le j = s+1-i \le s$, we reject the first $a_j$ applicants, wait until the decision of the $(i-1)^{th}$ selection (if $i \ge 2$), and then accept the next applicant whose qualification is better than all previously appeared applicants. Furthermore, our optimal strategy is right-hand based, i.e., the optimal strategies for two models with $s_1$ and $s_2$ selections in total ($s_1 < s_2$) share the same sequence $a_1, a_2, \ldots, a_{s_1}$ when it is viewed from the right. When the total number of applicants tends to infinity, our result agrees with the thresholds obtained by Gilbert and Mosteller. 

\end{abstract}

\section{Introduction}

The classical Secretary problem can be stated as follows: $N$ applicants apply for a single available Secretary position and the set of $N$ applicants can be ranked from the best to worst without ties according to their qualifications for the job. The applicants are presented sequentially and uniformly at random. Once an applicant is interviewed, an immediate decision is made on whether the person is accepted or rejected for the position and the decision cannot be revoked at a later time. Furthermore, only applicants observed so far can be used in the decision process. The problem of interest is to identify the best stopping rule, i.e., the rule that maximizes the probability of hiring the highest ranked applicant.

The Secretary problem was formally introduced by Gardner~\cite{G1, G2} and is considered a typical example in sequential analysis, optimization, and decision theory. Lindley~\cite{L1} established the best strategy using algebraic methods while Dynkin~\cite{D1} independently solved the problem by viewing the selection process as a Markov chain. For $N$ large enough, 
the solution to the problem turns out to be surprisingly simple: the employer needs to reject the first $N/e$ applicants, where $e$ is the base of the natural logarithm, and then accept the next applicant whose qualification is better than that of all previously observed applicants. The probability of identifying the best applicant tends to $1/e$ as $N$ tends to infinity.

The classical Secretary problem has been generalized in many directions~\cite{BIK1, BFGOR1, BFGOR2,CJMTUW1, F2, FW1, GM1, GM2, GKMN1, jones2020weighted, KS, LM1, LM2, MLM1, R1, soto}, including the Prophet inequality model. One such generalization, the Dowry problem (with multiple choices), introduced by Gilbert and Mosteller~\cite{GM2}, assumes that one is given a total of $s$ opportunities to select the best applicant, where $s \ge 1$. Using Dynkin's approach, Sakaguchi~\cite{S1} rigorously determined the optimal selection times for the Dowry setting. 

Motivated by recent works on learning with queries~\cite{AKB1,CPM1,MS1}, we introduce the problem of query-based sequential selection. In our new model, we assume that the decision making entity has a fixed number of $s-1$ opportunities to query an infallible expert. When an applicant is identified as the potentially highest ranked applicant after an exploration process, the expert provides an answer of the form ``the applicant is the best'' or ``the applicant is not the best''. If the answer is ``the applicant is the best'', then the sequential examination process terminates. However, if the expert responds ``the applicant is not the best'', a new exploration-exploitation stage begins as long as the query budget allows it. After the budget is exhausted, one is still allowed to make a final selection without consulting the expert. Therefore, with a budget of $s-1$ queries we can make at most $s$ selections. The goal is to find the strategy that maximizes the probability of selecting the best applicant.

Our new query-based model with the budget of $s-1$ queries is related to but different from the Dowry problem with $s$ selections. For both models, our results indicate that an optimal strategy is a $(a_s, a_{s-1}, \ldots, a_1)$-strategy, i.e., for the $i^{th}$ selection, where $1 \le i \le s$ and $1 \le j = s+1-i \le s$, we reject the first $a_j$ applicants, wait until the decision for the $(i-1)^{th}$ selection (if $i \ge 2$), and then accept the next applicant whose qualification is better than that of all previously seen applicants. Furthermore, the optimal strategies for the two models with $s_1$ and $s_2$ selections in total (w.l.o.g., $s_1 < s_2$) share the same sequence $a_1, a_2, \ldots, a_{s_1}$ when viewed from the right. When $N \to \infty$, our result agrees with the thresholds obtained by Gilbert and Mosteller~\cite{GM2}. On the other hand, the two models are very different from the perspective of the expected termination time, especially when the total number of selections is large. 

An important combinatorial method to study sequential problems under nonuniform ranking models was developed in a series of papers by Fowlkes and Jones~\cite{FJ19}, and Jones~\cite{jones2019,jones2020weighted}. For consistency, we use some of the notation and definitions from Jones~\cite{jones2020weighted} but also introduce a number of new concepts and combinatorial  proof techniques. In particular, finding recurrence relations for more than one selection is significantly more challenging than for the secretary problem, and the optimal strategies differ substantially from the classical ones as our results include multiple thresholds for stopping. 

The paper is organized as follows. Section~\ref{sec:preliminaries} introduces the relevant concepts, terminology and models used in the paper. The same section presents technical lemmas needed to establish our main results. Section~\ref{optimal} describes the optimal selection strategy, while Section~\ref{solve} describes the exact thresholds for the optimal strategy and the maximum probability of identifying the best applicant.

\section{Preliminaries}\label{sec:preliminaries}

The sample space is the set of all permutations of $N$ elements, i.e. the symmetric group $S_N$, with the underlying $\sigma$-algebra equal to the power set of $S_N$. The best applicant is indexed by $N$, the second-best applicant by $N-1,\ldots,$ and the worst applicant by $1$. In our model, there is a budget of $s-1$ queries ($s$ selections where $s \ge 1$). A permutation $\pi \in S_N$ is sampled from $S_N$ uniformly at random before the interview process. During the interview process, entries of $\pi$ are presented one-by-one from the left. The relative ordering of the presented positions of $\pi$ is the only information that can be used to decide whether to accept the current applicant.   


For a given permutation $\pi \in S_N$ drawn according to the uniform distribution, we say that a strategy \emph{wins} the game if it correctly identifies the best applicant within the $s$ selections when presented with $\pi$. The notion of a prefix is introduced to represent the current relative ordering of applicants.

\begin{defn}
Given a $\pi \in S_N$, the $k^{(th)}$ prefix of $\pi$, denoted by $\pi|_k,$ is a permutation in $S_k$ and it represents the relabelling of the first $k$ elements of $\pi$ according to their relative order. For example, if $\pi=[635124] \in S_6$ and $k=4$, then  $\pi|_4=[4231]$.
\end{defn}

\begin{defn}
Let $\sigma \in \bigcup\limits_{i = 1}^{N} S_i$ and assume that the length of the permutation, $|\sigma|$, equals $k$. A permutation $\pi \in S_N$ is said to be $\sigma$-prefixed if $\pi|_k = \sigma$. For example, $\pi=[165243] \in S_6$ is $\sigma=[1432]$-prefixed. 
\end{defn}

\begin{defn}
Let $\pi$ be $\sigma$-prefixed. We say that $\pi$ is {\em $\sigma$-winnable} if accepting the prefix $\sigma$, i.e. if accepting the $|\sigma|^{\text{th}}$ applicant when the order $\sigma$ is encountered identifies the best applicant with interview order $\pi$. More precisely, for $\sigma = [\sigma(1)\sigma(2) \cdots \sigma(k)]$, we have that $\pi$ is $\sigma$-winnable if $\pi$ is $\sigma$-prefixed and $\pi(k) = N$.
\end{defn}

\begin{defn}
A {\em left-to-right maximum} in a permutation $\pi \in S_N$ is a position whose value is the larger than all values to the left of the position. For example, if $\pi = [423516] \in S_6$, then the $1^{\text{st}}$, $4^{\text{th}}$, and $6^{\text{th}}$ position are left-to-right maximum,  since the value at each of the positions is larger than all values to their left.
\end{defn}

\begin{defn}
A permutation $\sigma \in \bigcup\limits_{i = 1}^{N} S_i$ is said to be {\em eligible} if it ends in a left-to-right maximum or has length $N$. For example, let $N = 6$, both $[1324]$ and $[165243]$ are eligible.
\end{defn}

Every strategy can be represented as a set of permutations (of possibly different lengths) that lead to an acceptance decision for the last applicant observed; such a set is called a {\em strike set}. More precisely, the selection process proceeds as follows: if the prefix we have seen so far is in the strike set, then we accept the current applicant and continue (provided a selection remains); if it does not belong to the strike set, we reject the current applicant and continue. For example, let $N=4$ and $s=1$. Then, the boxed set of permutations $A = \{[12], [213], [3124], [3214]\}$ in Figure~\ref{tree-1} is a strike set. The corresponding interview strategy may be summarized as follows: if the relative order of the applicants interviewed so far is in the set $A$, then accept the current applicant; otherwise, reject the current applicant. This turns out to be an optimal strategy with probability of successfully selecting the best applicant equal to $11/24$. Obviously, the strike set representing an optimal strategy only contains eligible permutations, since an optimal strategy only selects applicants that are left-to-right maximum. We also make use of $s$-strike sets defined below.

\begin{defn}
A set $X \subseteq \bigcup_{j = 1}^{N} S_j$ is called an {\em $s$-minimal set} if it is impossible to have $s+1$ elements $\alpha_1, \alpha_2, \ldots, \alpha_{s+1} \in X$ such that $\alpha_{i+1}$ is a prefix of $\alpha_i$, for all $i \in \{1, 2, \ldots, s\}$.
\end{defn}

\begin{defn}
A set of permutations $A \subseteq \bigcup\limits_{j = 1}^{N} S_j$ is called an {\em $s$-strike set} if it 

(i) It comprises prefixes that are eligible. 

(ii) The set $A$ is $s$-minimal. The set $A$ may contain elements $\alpha_1, \alpha_2, \ldots, \alpha_{s}$ such that $\alpha_{i+1}$ is a prefix of $\alpha_i$, for all $i \in \{1, 2, \ldots, s-1\}$.  In other words, based on an $s$-strike set one can make at most $s$ selections.

(iii) Every permutation in $S_N$ contains some element of $A$ as its prefix (i.e., given an $s$-strike set one can always make a selection based on its elements).	
\end{defn}


\begin{remark}
The $1$-strike set corresponds to the valid strike set defined in the paper of Jones~\cite{jones2020weighted} when only one choice is allowed. We use the term strike set whenever $s$, the number of total selections, is clear from the context.
\end{remark}

From the previous definition and the fact that we are allowed to make at most $s$ selections it follows that any optimal strategy for our problem can be represented by an $s$-strike set. A detailed analysis of this claim is provided in Lemma~\ref{expansion} and Theorem~\ref{winprob-2}.  For example, the set \{[1], [12], [213], [3124], [3214]\} in Figure~\ref{tree-1} is a $2$-strike set, which also represents an optimal strategy for the case $N = 4$ and $s=2$. See also Example~\ref{prefixtree}. 
Furthermore, for a permutation $\sigma$ of length $k$, where $1 \le k \le N$, and $i \in \{1, 2, \ldots, s\}$, we make use of the following probabilities:



$Q_i(\sigma)$: The probability of identifying the best
applicant with the strategy accepting the $k$-th position
and using the best strategy thereafter \textbf{conditioned on}
the pre-selected interviewing order $\pi$ being $\sigma$-prefixed and
$i$ selections still being available when interviewing the applicant at position $k$.

$Q_i^o(\sigma)$: The probability of identifying the best
applicant with the best strategy after making a decision
for the $k$-th position \textbf{conditioned on} the pre-selected interviewing order $\pi$ being $\sigma$-prefixed and $i$ selections still being available right after the interview of the $k$-th applicant.

$\bar{Q}_i(\sigma)$: $\bar{Q}_i(\sigma) = \max\{Q_i(\sigma), Q^o_i(\sigma)\}.$


Intuitively, $Q$ represents the probability of winning by accepting the current applicant while $Q^o$ is the probability of winning based on future selections in the interview process. In order to ensure the maximum probability of winning, an optimal strategy will examine two available choices, i.e. ``accept the current applicant'' or ``reject the current applicant and implement the best strategy in the future'' at each stage of the interview and select the one with a better chance of identifying the best applicant.

\begin{defn}
Let $\sigma$ be a permutation of length $k$, where $1 \le k \le N$. The {\em standard denominator} of $\sigma$, denoted by $SD(\sigma)$, is defined to be the number of  $\sigma$-prefixed permutations $ \pi \in S_N$. For simplicity, we denote the number of  $\sigma$-winnable permutations $ \pi \in S_N$ by $Win(\sigma)$.
\end{defn}

\begin{defn}
The $\xor$ operation for $\frac{a}{b}$ and $\frac{c}{d}$ is defined as $\frac{a}{b} \xor \frac{c}{d} = \frac{a+c}{b+d}$. It is used to compute the probability of the union of two disjoint events from two disjoint sample spaces over a new sample space equal to the union of the sample spaces.
\end{defn}

\begin{defn}\label{children}
	For each $\sigma$ of length $\ell-1$, where $2 \le \ell \le N$, we define $\lambda_j(\sigma)$, $1 \le j \le \ell$, to be the $\sigma$-prefixed permutation of length $\ell$ such that its last position has value $j$ after relabelling according to the first $\ell-1$ positions of $\sigma$. For example, for $\sigma = [123],$ a permutation of length $3$, we have $\lambda_1(\sigma) = [2341], \lambda_2(\sigma) = [1342], \lambda_3(\sigma) = [1243]$ and $\lambda_4(\sigma) = [1234]$.
\end{defn}


Let $\sigma$ be a permutation of length $1 \le k \le N$ with $Q_i(\sigma), Q^o_i(\sigma), \bar{Q}_i(\sigma)$ defined as above with $1 \leq i \leq s$ selections available right before processing the $k$-th applicant of a $\sigma$-prefixed permutation. For each $1 \le i \le s$, if the $k$-th position of $\sigma$ is selected, then the number of selections available decrease by one; if the $k$-th applicant is rejected, then the number of selections available does not change. When the number of available selections becomes zero or all applicants are examined, the process terminates.

\begin{remark}\label{key}
The case when $i = 1$ is a corollary of a theorem by Jones~\cite{jones2020weighted}. Note that we do not simplify the fractions in the expressions for the probabilities $Q,Q^o, \bar{Q}$ until we solve the problem (for example, if the numerator and denominator of $Q_1(\sigma) = \frac{\text{Win}(\sigma)}{\text{SD}(\sigma)}$ have a common divisor $d$, we do not cancel $d$). After making a decision on the $|\sigma|$-th applicant, the interviewer examines the next applicant while the relative order of the interviewed applicants changes to one of $\lambda_1(\sigma), \ldots, \lambda_{k+1}(\sigma)$. An optimal strategy involves making a decision with the largest probability of winning when encountering each of the $\lambda_1(\sigma), \ldots, \lambda_{k+1}(\sigma)$. Thus, 

\begin{equation}\label{basicq0o}
Q^o_1(\sigma) = \bar{Q}_1(\lambda_1(\sigma)) \xor \cdots \xor \bar{Q}_1(\lambda_{k+1}(\sigma)) \quad \text{ and } \quad SD(\sigma) = \sum\limits_{j = 1}^{k+1} SD(\lambda_{j}(\sigma)).
\end{equation}
\end{remark}

For each $i \in \{2, \ldots, s\}$, we know by the definitions of $Q_i, Q^o_i$ and $\bar{Q}_i$ that
\begin{equation}\label{qqoqbar}
\bar{Q}_i(\sigma) = \max\{Q_i(\sigma), Q_i^o(\sigma)\} \text{ and }Q_i(\sigma) = Q_{1}(\sigma) + Q_{i-1}^o(\sigma).
\end{equation}

The second equation of~\eqref{qqoqbar} holds since there are two (disjoint) events that ensure winning after examining the current applicant, i.e., (a) the current applicant is the best and (b) the current applicant is not the best but we identify the best applicant at a later time with a best strategy after rejecting the current applicant. In the first case, the probability of successfully identifying the best applicant is $Q_1(\sigma)$; in the second case, the number of available selections decreases by one and the corresponding probability is $Q_{i-1}^o(\sigma)$.

\begin{proposition}\label{basicexpansion}
Let $\sigma$ be a permutation of length $k$ with $1 \le k \le N-1$ and let $2 \le i \le s$. Then 
\begin{equation}\label{basicqio}
Q^o_i(\sigma) = \bar{Q}_i(\lambda_1(\sigma)) \xor \cdots \xor \bar{Q}_i(\lambda_{k+1}(\sigma)).
\end{equation}
\end{proposition}
\begin{proof}
The case of $i = 1$ is covered by~\eqref{basicq0o}. By the definition of the $Q_i^o$ probabilities, we know that $Q_i^o(\sigma)$ can be expressed as a fraction with denominator $SD(\sigma)$ and the numerator equal to the number of $\sigma$-prefixed permutations such that the best applicant can be selected using an optimal strategy right after making a decision at the $|\sigma|^{\text{th}}$ applicant in $\pi$, conditioned on the event that there are $i$ available selections left after making a decision for the $|\sigma|^{\text{th}}$ applicant in $\pi$. Since we will see one of $\lambda_1(\sigma), \ldots, \lambda_{k+1}(\sigma)$ after making a decision at the $k^{\text{th}}$ applicant in $\pi$ and an optimal strategy would make the decision with largest probability of winning at each of $\lambda_1(\sigma), \ldots, \lambda_{k+1}(\sigma)$, the stated equation~\eqref{basicqio} holds.
\end{proof}


Following a methodology suggested by Jones~\cite{jones2020weighted}, in our proof we make extensive use of \textit{prefix trees} which naturally capture the inclusion-relationship between prefixes of permutations. The concept is best described by an illustrative example, depicted for the case of $S_4$ in Fig.\,\ref{tree-1}. The correspondence between sub-trees/sub-forests is crucial for the proof of Lemma~\ref{only length}. For example, let $F_1$ be the sub-forest obtained by deleting the vertex $[12]$ in the tree induced by $[12]$ and its children, let $F_2$ be the sub-forest obtained by deleting the vertex $[21]$ in the tree induced by $[21]$ and its children. Then, there is a bijection between $F_1$ and $F_2$ which preserves all the probabilities used in evaluating optimal strategies.

\begin{defn}
Let $T$ be the tree formed by the inclusion-relationship between prefixes of permutations of length at most $N$. In other words, for $V$ being the collection of all permutations of length at most $N$, we let $T = (V, E)$ be such that if $\sigma, \tau \in V$ and $\sigma$ is a prefix of $\tau$ with $|\sigma| = |\tau| - 1$, then we have an edge $\sigma \tau \in E$. We define $\bar{T}(\sigma)$ to be the subtree in $T$ comprising $\sigma$ and its children and let $T^o(\sigma) = \bar{T}(\sigma) - \sigma$ be the forest obtained by deleting $\sigma$ from $\bar{T}(\sigma)$. For example, in Figure~\ref{tree-1}, if $\sigma = [12]$, then $\bar{T}(\sigma)$ is the subtree induced by the vertices $$\{[12], [123], [132], [231], [1234], [1243], [1342], [2341], [1324], [1423], [1432], [2431], [2314], [2413], [3412], [3421]\}$$  and  $T^o(\sigma)$ is  the forest induced by the vertices $$\{[123], [132], [231], [1234], [1243], [1342], [2341], [1324], [1423], [1432], [2431], [2314], [2413], [3412], [3421]\}.$$
\end{defn}

\begin{defn}
We say that a prefix $\sigma$ is {\em type $i$-positive} if $Q_i(\sigma) \ge Q^o_i(\sigma)$ and {\em type $i$-negative} otherwise, where $i \in \{1, \ldots, s\}$.
\end{defn}

On Proposition~\ref{probs-qiqio}, we show that the probabilities $Q_i^o, Q_i, \bar{Q}_i$ for each $i \in \{1, \ldots, s\}$ can be calculated (pre-calculated) using a sequential procedure (backward induction). It is crucial and will be used to show the optimal strategy is a positional strategy with $s$ thresholds. After knowing this fact, we will find the winning probability in Section~\ref{solve} by solving a recurrence relation instead.

\begin{proposition}\label{probs-qiqio}
Let $\tau$ be any permutation of length $k$, where $1 \le k \le N$. The probabilities $Q_i^o(\tau), Q_i(\tau), \bar{Q}_i(\tau)$ for each $i \in \{1, \ldots, s\}$ can be calculated using a sequential procedure.
\end{proposition}

\begin{proof}
In order to compute the probabilities $Q_i^o(\tau), Q_i(\tau)$ for each $i \in \{1, \ldots, s\}$, we use double induction on the subscript $i$ and the length of the prefix. 

\textbf{Base case for outer induction on $i$:} We first show the base case which corresponds to $i = 1$. We first observe that the prefixes of length $N$ are type $1$-positive, which serves as a base case for inner induction on the length of a prefix. More precisely, for a permutation $\tau$ of length $N$, if $\tau(N) = N$ then $Q_1(\tau) = \bar{Q}_1(\tau) = 1$ and $Q^o_1(\tau) = 0$; if $\tau(N) < N$ then $Q_1(\tau) = Q^o_1(\tau) = \bar{Q}_1(\tau) = 0$. 

Assume that the probabilities $Q_1, Q^o_1, \bar{Q}_1$ for permutations of length longer than $k$, $1 \le k \le N-1$, are already known. We show that $Q_1^o(\tau), Q_1(\tau)$, and $\bar{Q}_1(\tau)$ can be calculated, where now $\tau$ is a permutation of length $k$. We know the value of $Q_1(\tau)$ can be obtained by finding a fraction where the denominator equal to the number of $\tau$-prefixed permutations of length $N$ and the numerator equal to the number of $\tau$-winnable permutations of length $N$, which can be obtained by counting permutations of length $N$ that has $\tau$ as a prefix and  $\tau$-prefixed permutations with value $N$ on the $|\tau|^{\text{th}}$ position. By~\eqref{basicq0o}, the probability $Q^o_1(\tau)$ can be obtained from $Q^o_1(\tau) = \xor_{j = 1}^{k+1} \bar{Q}_1(\lambda_{j}(\tau))$. Since each $\lambda_{j}(\tau)$ has length larger than that of $\tau$, each of $ \bar{Q}_1(\lambda_{j}(\tau))$ is already known by the inductive hypothesis. The $\bar{Q}_1(\tau)$ probabilities can be determined from $\bar{Q}_1(\tau) = \max\{Q_1(\tau), Q_1^o(\tau)\}$, where $Q_1(\tau)$ and $Q_1^o(\tau)$ are already obtained by previous arguments.

\textbf{Main proof after the base case:} Assume now that we know for each $1 \le q \le i-1$, $i \in \{1, \ldots, s\}$, the $Q_q, Q^o_q, \bar{Q}_q$ probabilities for every permutation in $\bigcup\limits_{j = 1}^{N} S_j$. We prove the result for $i$. The probabilities $Q_i, \bar{Q}_i$ of prefixes of length $N$ take either the value $1$ or $0$ depending on whether the last position has value $N$, while the probabilities $Q^o_i$ are all $0$. This serves as the base case for the inner induction argument on the length of the prefixes. 

Let $\tau$ be a permutation of length $k$, where $1 \le k \le N-1$. Given the probabilities $Q_i, Q_i^o, \bar{Q}_i$ for prefixes of lengths greater than $k$, the $Q_i^o(\tau)$ probabilities for prefixes $\tau$ of length $k$ can be obtained by $Q_i^o(\tau) = \xor_{j = 1}^{k+1} \bar{Q}_i(\lambda_{j}(\tau))$ from Proposition~\ref{basicexpansion}, where the probabilities $\bar{Q}_i(\lambda_{j}(\tau))$ are known by the inductive hypothesis. Moreover, we have $Q_i(\tau) = Q_1(\tau) + Q_{i-1}^o(\tau)$, and all these probabilities are available by the inductive hypothesis and from the base case. Furthermore, $\bar{Q}_i(\tau) = \max\{Q_i(\tau), Q^o_i(\tau) \}$.
\end{proof}

Recall that $Q_1(\sigma)$ can be written as a fraction with denominator $SD(\sigma)$ and the numerator equal to the number of permutations $\pi \in S_N$ that are $\sigma$-winnable. In Propositions~\ref{easyexpansion} and Lemma~\ref{expansion} below we show that the probabilities $Q_i(\sigma)$, where $2 \le i \le s$, and $Q^o_i(\sigma)$, where $1 \le i \le s$, can also be expressed as fractions with the standard denominator $SD(\sigma)$ (and then evaluate the numerator). To this end, note that for a permutation $\sigma$ of length at least $N-i$, with $0 \le i \le s-1$, 
$Q_{i+1}(\sigma) = \ldots = Q_{s}(\sigma)$. Furthermore, for a permutation $\sigma$ of length at least $N-i$, with $1 \le i \le s-1$, $Q^o_{i}(\sigma) = Q_i^o(\sigma) = \ldots = Q_{s}^o(\sigma)$.

We first show in Proposition~\ref{easyexpansion} that $Q_i^o(\sigma)$ can be expressed as a special sum (using the $\xor$ notation) of the $Q_i$ probabilities of some longer permutations. We express it using the basic probability $Q_1$ in Lemma~\ref{expansion}.

\begin{proposition}\label{easyexpansion}
For each $1 \le i \le s$ and a permutation $\sigma$ of length $\ell$ with $\ell \le N-1-i$, there exists a collection of $\sigma$-prefixed permutations $\Gamma_{\sigma, i}$ such that each $\mu \in \Gamma_{\sigma, i}$ is of length larger than $|\sigma|$ and type $i$-positive, the set $\Gamma_{\sigma, i}$ is $1$-minimal, and $Q^o_i(\sigma) = \xor_{\mu \in \Gamma_{\sigma, i}} Q_i(\mu)$, i.e.,
\begin{equation}\label{qiqo-1}
Q^o_i(\sigma) \cdot SD(\sigma) = \sum\limits_{\mu \in \Gamma_{\sigma, i}} Q_i(\mu) \cdot SD(\mu) \quad \text{ and } \quad SD(\sigma) = \sum\limits_{\mu \in \Gamma_{\sigma,i}} SD(\mu). 
\end{equation}
By~\eqref{qqoqbar}, we know that
\begin{equation}\label{qisd-1}
Q_i(\sigma) \cdot SD(\sigma) = Q_1(\sigma) \cdot SD(\sigma) + \sum\limits_{\mu \in \Gamma_{\sigma, i-1}} Q_{i-1}(\mu) \cdot SD(\mu).
\end{equation}
\end{proposition}

\begin{proof}
The case $i = 1$ of Equation~\eqref{qiqo-1} was covered in the work of Jones~\cite{jones2020weighted} (described using the $\xor$ notation). Since $Q_2(\sigma) = Q_1(\sigma) + Q^o_1(\sigma)$, there is a set $\Gamma_{\sigma,1}$ such that
$$Q_2(\sigma) \cdot SD(\sigma) = Q_1(\sigma) \cdot SD(\sigma) + Q^o_1(\sigma) \cdot SD(\sigma) =  Q_1(\sigma) \cdot SD(\sigma) + \sum\limits_{\mu \in \Gamma_{\sigma,1}} Q_1(\mu) \cdot SD(\mu),$$
where $\Gamma_{\sigma,1}$ is $1$-minimal and consists of type $1$-positive permutations of length larger than $|\sigma|$.

After making a decision on the $|\sigma|^{\text{th}}$ applicant, an optimal strategy will examine the children of $\sigma$, i.e. $\lambda_{1}(\sigma), \ldots, \lambda_{\ell+1}(\sigma)$, in the prefix tree and then make a decision with the largest probability of winning. We describe the following algorithm to prove the proposition pertaining to $Q^o_i(\sigma)$. 

\begin{itemize}
\item[Initialization step:] Let $\Gamma_i = \varnothing$ and $B = \{\lambda_{1}(\sigma), \ldots, \lambda_{\ell+1}(\sigma)\}$. 

We next repeat the main step until the process terminates.

\item[Main step:] Check if $B = \varnothing$; if yes, stop and return the set $\Gamma_i$; if no, then do the following: Pick an arbitrary permutation $\phi \in B$, say of length $q$, with $|\sigma| < q \le N$; check if $\phi$ is both eligible and type $i$-positive ($Q_i(\phi) \ge Q_i^o(\phi)$); if yes, set $\Gamma_i = \Gamma_i \cup \phi$ and $B = B - \phi$; if no, do not update $\Gamma_i$ and let $B = (B - \phi) \cup \bigcup\limits_{j = 1}^{q+1} \lambda_{j}(\phi)$. Note that the probabilities $Q_i(\phi)$ and $Q_i^o(\phi)$ are known by Proposition~\ref{probs-qiqio}.
\end{itemize}

Since the permutations of length $N$ are type $i$-positive for each $1 \le i \le s$, the algorithm eventually terminates. By the criteria on the main step of the algorithm, it will produce a set $\Gamma_i$ of type $i$-positive eligible permutations that is $1$-minimal and each of the $\gamma \in \Gamma_i$ has length larger than $|\sigma|$. At the end of the process, $B$ is an empty set. To see this, we make the following two observations.

\textbf{Observation (i):} There is no pair of elements $\alpha,\beta \in \Gamma_i$ such that $\alpha$ is a prefix of $\beta$, i.e., $\Gamma_i$ contains $1$-minimal prefixes, since otherwise the sub-forest $T^o(\alpha)$ will not be processed by the algorithm and it will be impossible for $\beta$ to be selected for inclusion in $\Gamma_i$.

\textbf{Observation (ii):} Since we choose a permutation only if it is type $i$-positive and eligible, every permutation in $\Gamma_i$ is type $i$-positive and eligible. 

Furthermore, by the main step of the algorithm and the induction hypothesis, 
$$Q^o_i(\sigma) \cdot SD(\sigma) = \sum\limits_{\gamma \in \Gamma_i} Q_i(\gamma) \cdot SD(\gamma) \quad \text{ and } \quad SD(\sigma) = \sum\limits_{\gamma \in \Gamma_i} SD(\gamma).$$
Equivalently, if we divide by $SD(\sigma)$ on both sides, then
\begin{equation}\label{expansion-2}
Q^o_i(\sigma) = \xor_{\gamma \in \Gamma_i} Q_i(\gamma).
\end{equation}

To prove the corresponding formula for $Q_{i+1}(\sigma)$, note that $Q_{i+1}(\sigma) = Q_1(\sigma) + Q^o_i(\sigma)$ and invoke the result of~\eqref{qiqo-1} for $Q_i^o(\sigma)$. 
\end{proof}


\begin{lemma}\label{expansion}
For each $1 \le i \le s$ and a permutation $\sigma$ of length $\ell$ with $\ell \le N-1-i$, there exists a collection of $\sigma$-prefixed permutations $G^{\sigma, i}$ that can be partitioned into $\Gamma_1^{\sigma, i}, \ldots, \Gamma_i^{\sigma, i}$ such that each $\mu \in G^{\sigma, i}$ is of length larger than $|\sigma|$, the set $G^{\sigma, i}$ is $i$-minimal, each $\gamma \in \Gamma_j^{\sigma, i}$ is type $j$-positive where $1 \le j \le i$, and
\begin{equation}\label{qisd-2}
Q_i(\sigma) \cdot SD(\sigma) = Q_1(\sigma) \cdot SD(\sigma) + \sum\limits_{\mu \in G^{\sigma, i-1}} Q_1(\mu) \cdot SD(\mu), \quad \text{ and }
\end{equation}
\begin{equation}\label{qiqo-2}
Q^o_i(\sigma) \cdot SD(\sigma) = \sum\limits_{\gamma \in \Gamma^{\sigma, i}_i} Q_i(\gamma) \cdot SD(\gamma)
= \sum\limits_{\mu \in G^{\sigma, i}} Q_1(\mu) \cdot SD(\mu) \quad \text{ and } \quad SD(\sigma) = \sum\limits_{\gamma \in \Gamma^{\sigma,i}_i} SD(\gamma). 
\end{equation}
\end{lemma}
\begin{remark}
The $G^{\sigma, i}$ set for the probability $Q^o_i(\sigma)$ can be viewed as the strike set representing an optimal strategy for the problem of how to select the best applicant right after we made a decision on the $|\sigma|$th applicant and still have $i$ selections left. Each $\Gamma_j^{\sigma, i}$ can be viewed as the $j$th layer of the poset structure of $G^{\sigma, i}$ (formed by the prefix-relationship between permutations). Note that Equation~\eqref{qisd-2} implies that the numerator of $Q_i(\sigma)$, i.e. $Q_i(\sigma) \cdot SD(\sigma)$, can be interpreted as the sum of the number of $\sigma$-winnable permutations in $S_N$ and the number of $\mu$-winnable permutations in $S_N$, where $\mu \in G^{\sigma, i-1}$. Equation~\eqref{qiqo-2} implies the numerator of $Q_i^o(\sigma)$ can be interpreted in a similar manner.
\end{remark}
\begin{proof}
The case $i = 1$ of Equation~\eqref{qiqo-2} was covered in the work of Jones~\cite{jones2020weighted} (described using the $\xor$ notation). The case $i = 2$ of equation~\eqref{qisd-2} was already proved in Proposition~\ref{easyexpansion}.


Next, we assume the claim is true for $Q_i(\sigma)$ and $Q^o_{i-1}(\sigma)$, where $2 \le i \le s$, and then first prove~\eqref{qiqo-2} for $Q^o_i(\sigma)$ and then~\eqref{qisd-2} for $Q_{i+1}(\sigma)$ (except when $i = s$ since $Q_{s+1}(\sigma)$ is not defined). 

By Proposition~\ref{easyexpansion}, there is a collection $\Gamma_i$ of permutations longer than $\sigma$ and satisfy that 
$$Q^o_i(\sigma) \cdot SD(\sigma) = \sum\limits_{\gamma \in \Gamma_i} Q_i(\gamma) \cdot SD(\gamma) \quad \text{ and } \quad SD(\sigma) = \sum\limits_{\gamma \in \Gamma_i} SD(\gamma).$$

By the induction hypothesis, for each $\gamma \in \Gamma_i$ there exists a set $G^{\gamma, i-1}$ of $(i-1)$-minimal permutations of length larger than $|\gamma|$, which can be partitioned into $\Gamma^{\gamma,i-1}_1 \cup \ldots \cup \Gamma^{\gamma, i-1}_{i-1}$. In this case, each $\gamma_j \in \Gamma^{\gamma, i-1}_j$ is type $j$-positive, where $1 \le j \le i-1$, and
$$Q_i(\gamma) \cdot SD(\gamma) = Q_1(\gamma) \cdot SD(\gamma) + \sum\limits_{\mu \in G^{\gamma, i-1}} Q_1(\mu) \cdot SD(\mu).$$ 

Let $\Gamma^{\sigma, i}_i = \Gamma_i$ and $\Gamma^{\sigma, i}_{j} = \bigcup\limits_{\gamma \in \Gamma_i} \Gamma^{\gamma, i-1}_j,$ for $1 \le j \le i-1$. The set $G^{\sigma, i} := \bigcup\limits_{j = 1}^{i} \Gamma^{\sigma, i}_j$ comprises $i$-minimal permutations, and there is a partition of $G^{\sigma, i}$ into $\Gamma^{\sigma, i}_1, \ldots, \Gamma^{\sigma, i}_i$ such that each $\gamma \in \Gamma^{\sigma, i}_j$ is type $j$-positive, where $1 \le j \le i$, and equation~\eqref{qiqo-2} holds, i.e.,
 $$Q^o_i(\sigma) \cdot SD(\sigma) = \sum\limits_{\gamma \in \Gamma^{\sigma, i}_i} Q_i(\gamma) \cdot SD(\gamma)
= \sum\limits_{\mu \in G^{\sigma, i}} Q_1(\mu) \cdot SD(\mu).$$

To prove the corresponding formula for $Q_{i+1}(\sigma)$, note that $Q_{i+1}(\sigma) = Q_1(\sigma) + Q^o_i(\sigma)$ and invoke the result of equation~\eqref{qisd-2}. 
\end{proof}

\begin{theorem}\label{winprob-2}
There exists a strike set $A$ which can be partitioned as $A_s \cup \cdots \cup A_1,$  where each $A_i$ is a set of type $i$-positive $1$-minimal permutations , for $1 \le i \le s$,  and the maximum probability of winning equals
$\xor_{\sigma \in A_s} Q_s(\sigma)$. 
Furthermore, the probability of winning equals 
$$\frac{\sum\limits_{\sigma \in A} Q_1(\sigma) \cdot SD(\sigma)}{N!}.$$
\end{theorem}

\begin{proof}
The optimal winning probability is equal to $\bar{Q}_s([1])$. The proof of Lemma~\ref{expansion} for $\sigma = [1]$ establishes the proof of the theorem. If $Q_s([1]) \ge Q^o_s([1])$, then the strike set $A$ corresponds to the set $G^{[1], s-1} \cup \{[1]\}$ in equation~\eqref{qisd-2} of Lemma~\ref{expansion}; if $Q_s([1]) < Q^o_s([1])$, then the strike set $A$ corresponds to the set $G^{[1], s}$ in equation~\eqref{qiqo-2} of Lemma~\ref{expansion}.
\end{proof}

\begin{defn}
Define $g_\tau$ to be an action on the symmetric group that arranges (permutes) a prefix $\sigma = [12 \cdots k]$ to some other prefix $\tau$ of the same length $k$. This action can be extended to $\bar{T}(\sigma)$ and is denoted by $g_\tau \cdot \pi$: It similarly permutes the first $k$ entries and fixes the remaining entries of $\pi \in \bar{T}(\sigma)$. It is easy to see that $g_\tau$ is a bijection from $\bar{T}(\sigma)$ to $\bar{T}(\tau)$.
\end{defn}

\begin{example}
Let $N = 6$, $\tau = [132]$, and $\pi = [245361]$. Clearly, $\pi$ is $[123]$-prefixed. Then $g_{\tau} \cdot \pi = [254361]$.
\end{example}

Consider an arbitrary permutation $\sigma \in S_N$, and let $1 \le i \le s$. 

\textbf{Case 1:} The last position of $\sigma$ is $N$. In this case we have $Q_i(\sigma) = 1$, $Q^o_i(\sigma) = 0$, and $\bar{Q}_i(\sigma) = 1$. 

\textbf{Case 2:} The last position of $\sigma$ is not $N$. In this case we have $Q_i(\sigma) = 0$, $Q^o_i(\sigma) = 0$, and $\bar{Q}_i(\sigma) = 0$.\\

To proceed with our analysis, we invoke the following result.






\begin{lemma}\label{qi}
For all prefixes $\tau$, the probabilities $Q_i$ for each $i \in \{1, \ldots, s\}$ are preserved under the restricted bijection $g_{\tau}: T^o(12 \cdots k) \to T^o(\tau)$, where $k \le N-1$ is the length of $\tau$. If $\tau$ is eligible then $Q_i([12 \cdots k]) = Q_i(\tau)$. Consequently, for $\sigma \in T^o(12 \cdots k)$ (and, additionally, for $\sigma = [12 \cdots k]$ if $\tau$ is eligible),

(a) The probabilities $Q_i^o(\sigma)$ are preserved by $g_{\tau}$;

(b) The probabilities $\bar{Q}_i(\sigma)$ are preserved by $g_{\tau}$;

(c) If $\sigma$ and $\tau$ are eligible, we have that $\sigma$ is type $i$-positive if and only if $g_{\tau} \cdot \sigma$ is type $i$-positive. 

\end{lemma}

\begin{proof}
The case when $i = 1$ was a corollary of Theorem 3.5 of Jones~\cite{jones2020weighted}. The proof proceeds by induction on the subscript $i$ of the probabilities $Q_i, Q^o_i, \bar{Q}_i$. Assume that the result holds for the probabilities $Q_m, Q^o_m, \bar{Q}_m$ with $m \le i-1$. We prove the claimed result for $Q_i, Q^o_i, \bar{Q}_i$, where $2 \le i \le s$. Let $\sigma \in T^o(12 \cdots k)$, we have
$$Q_i(g_{\tau} \cdot \sigma) = Q_1(g_{\tau} \cdot \sigma) + Q_{i-1}^o(g_{\tau} \cdot \sigma) = Q_1(\sigma) + Q_{i-1}^o(\sigma) = Q_i(\sigma).$$ 
If $\tau$ is eligible then we apply the argument to the restricted bijection $T^o([12 \cdots (k-1)]) \to T^o(\tau|_{k-1})$. 

\begin{claim}\label{Q^o_i}
For $\sigma$ of length $\ell$, where $k < \ell \le N$, $Q^o_i(g_{\tau} \cdot \sigma) = Q_i^o(\sigma)$. 
\end{claim}
\begin{proof}
We use induction on the length of $\sigma$. When $\sigma$ has length $N$ it holds that $Q^o_i(g_{\tau}\cdot \sigma) = Q^o_i(\sigma) = 0 $.  Assume now that statement (a) holds for all permutations of length at least $\ell+1$. We next present an argument for the case when $\sigma$ is of length $\ell$, where $k < \ell \le N-1$.

By the already proved statement for the probability $Q_i$ and the induction hypothesis, the probabilities $Q_i$ and $Q^o_i$ for a permutation $\mu$ of length larger than $\ell$ only depend on the length of $\mu$ and the value of the last position of $\mu$. By the Main Step of the algorithm described in the proof of Lemma~\ref{expansion}, if we process $\sigma$ and end up obtaining a set $\Gamma_i = \{\gamma_1, \ldots, \gamma_{r}\}$, then when we process $g_{\tau} \cdot \sigma$ we end up obtaining the set $\Gamma_i' = \{g_{\tau} \cdot \gamma_1, \ldots, g_{\tau} \cdot \gamma_{r}\}$. Therefore, by Lemma~\ref{expansion}, 
$$
Q_i^o(g_{\tau} \cdot \sigma) = Q_i(g_{\tau} \cdot \gamma_1) \xor Q_i(g_{\tau} \cdot \gamma_2) \xor \cdots \xor Q_i(g_{\tau} \cdot \gamma_r)
$$
$$
= Q_i(\gamma_1) \xor Q_i(\gamma_2) \xor \cdots \xor \cdot Q_i(\gamma_r) = Q^o_i(\sigma). \qedhere
$$
\end{proof}

By Claim~\ref{Q^o_i}, statement (a) is true. Claims (b) and (c) can be established from the previous results and the fact $\bar{Q}_i(\sigma) = \max\{Q_i(\sigma), Q^o_i(\sigma)\}$. This completes the main part of the proof.

Each of statements (a), (b), and (c) hold for $\sigma = [12 \cdots k]$ and a permutation $\tau$ which is eligible and of length $k$, since we can apply the corresponding argument to the restricted bijection $T^o([12 \cdots (k-1)]) \to T^o(\tau|_{k-1})$.
\end{proof}

For $1 \le i \le s$, we know from Lemma~\ref{qi} that the probabilities $Q_i(\sigma), Q^o_i(\sigma), \bar{Q}_i(\sigma)$ only depend on the length and the value of the last position of $\sigma$. Moreover, we prove next that $Q^o_i(\sigma)$ only depends on the length of $\sigma$ and does not depend on the value of the last position of $\sigma$. 

\begin{lemma}\label{only length}
For all $1 \le i \le s$, the probability $Q^o_i(\sigma)$ only depends on the length of $\sigma$.
\end{lemma}

\begin{proof}
We know $Q_i^o(\sigma) = 0$ for every $\sigma$ of length $N$. Let $\sigma'$ and $\sigma''$ be two permutations of length $k-1$, where $2 \le k \le N$. For each permutation $\sigma$ of length $k-1$, recall the definition of $\lambda_j(\sigma)$, $1 \le j \le k$, which was defined in Definition~\ref{children}. Let $\phi = [12 \cdots (k-1)]$. By Lemma~\ref{qi}, and using the bijections $g_{\sigma'}: T^o(\phi) \to T^o(\sigma')$ and $g_{\sigma''}: T^o(\phi) \to T^o(\sigma'')$, we have 
$$Q^o_i(\sigma') = \bar{Q}_i(\lambda_{1}(\sigma')) \xor \cdots \xor \bar{Q}_i(\lambda_{k}(\sigma')) = \bar{Q}_i(\lambda_{1}(\phi)) \xor \cdots \xor \bar{Q}_i(\lambda_{k}(\phi))$$
$$ = \bar{Q}_i(\lambda_{1}(\sigma'')) \xor \cdots \xor \bar{Q}_i(\lambda_{k}(\sigma'')) = Q^o_i(\sigma''). \qedhere$$

\end{proof}

In order to simplify our exposition, we henceforth change the notation and let $Q_i(\sigma)$, $Q_i^o(\sigma)$, $\bar{Q}_i(\sigma)$ stand for the \emph{numerators} in the definition of the underlying probabilities, each with respect to the standard denominator $SD(\sigma)$.

Next, let $\sigma = [12 \cdots (k-1)]$, recall that $\lambda_{j}(\sigma) = [12 \cdots (j-1)(j+1) k j]$, and 
$\lambda_{k}(\sigma) = [12 \cdots k]$, with $2 \le k \le N$. Let $\bar{Q}_0$ be defined to be $0$ for all permutations. We have the following relations.



\begin{lemma}\label{relation-2}
For $2 \le i \le s$, one has 
$$Q_i^o(\sigma) = (k-1) \cdot Q_i^o(\lambda_{k}(\sigma)) + \bar{Q}_i(\lambda_{k}(\sigma))  \quad \text{and} \quad Q_i(\sigma) = (k-1) \cdot Q_i(\lambda_{k}(\sigma)) + \bar{Q}_{i-1}(\lambda_{k}(\sigma)).$$
\end{lemma}

\begin{proof}
We first consider $Q_i^o(\sigma)$. We know that $\sigma$ has $k$ children in the prefix tree, and these children are $\lambda_{1}(\sigma), \ldots, \lambda_{k}(\sigma)$. The permutation $\lambda_{k}(\sigma)$ itself is an eligible child thus $\bar{Q}_i(\lambda_{k}(\sigma))$ is the optimal probability for the subtree rooted at $\lambda_{k}(\sigma)$. 




Therefore, since $\lambda_1(\sigma), \ldots, \lambda_{k-1}(\sigma)$ are not eligible and by Lemma~\ref{only length},
$$Q_i^o(\sigma) = \bar{Q}_i(\lambda_{1}(\sigma)) + \ldots + \bar{Q}_i(\lambda_{k-1}(\sigma)) + \bar{Q}_i(\lambda_{k}(\sigma)) $$
\begin{equation}\label{qio-expansion}
= Q^o_i(\lambda_{1}(\sigma)) + \ldots + Q^o_i(\lambda_{k-1}(\sigma)) + \bar{Q}_i(\lambda_{k}(\sigma)) = (k-1) \cdot Q_i^o(\lambda_{k}(\sigma)) + \bar{Q}_i(\lambda_{k}(\sigma)).    
\end{equation}

By Lemma~\ref{relation-2} and~\eqref{qio-expansion}, 
$$Q_i(\sigma) = Q_1(\sigma) + Q_{i-1}^o(\sigma) = (k-1) \cdot Q_1(\lambda_k(\sigma)) + (k-1) \cdot Q_{i-1}^o(\lambda_k(\sigma)) + \bar{Q}_{i-1}(\lambda_k(\sigma))$$
$$= (k-1) \cdot Q_i(\lambda_k(\sigma)) + \bar{Q}_{i-1}(\lambda_k(\sigma)). \qedhere$$

\end{proof}


By Lemma~\ref{relation-2}, if an eligible permutation is negative then all eligible permutations of shorter length are negative as well.

\begin{corollary}\label{cor-2}
For increasing prefixes $\sigma = [12 \cdots (k-1)]$ and $\lambda_{k}(\sigma) = [12 \cdots k]$, we have that if $\lambda_{k}(\sigma)$ is type $i$-negative then $\sigma$ is type $i$-negative, where $1 \le i \le s$.
\end{corollary}

\begin{proof}
Suppose $\sigma$ is type $i$-negative, i.e., $Q^o_i(\lambda_{k}(\sigma))>Q_i(\lambda_{k}(\sigma))$. By Lemma~\ref{relation-2} and $\bar{Q}_{i}(\lambda_{k}(\sigma)) \ge \bar{Q}_{i-1}(\lambda_{k}(\sigma))$,
$$Q_i^o(\sigma) = (k-1) \cdot Q_i^o(\lambda_{k}(\sigma)) + \bar{Q}_i(\lambda_{k}(\sigma)) >  (k-1) \cdot Q_i(\lambda_{k}(\sigma))  + \bar{Q}_{i-1}(\lambda_{k}(\sigma)) = Q_i(\sigma). \qedhere$$
\end{proof}


In words, Corollary~\ref{cor-2} asserts that each $Q_i(k) - Q_i^o(k)$ is a non-decreasing function of $k$.

\begin{example}\label{prefixtree}
To clarify the above observations and concepts, we present an example for the case $s = 2$ and $N = 4$. An optimal strategy is the $(0,1)$-strategy where we accept the first applicant, ask the Genie whether this applicant is the best, and then accept the next left-to-right maximum. The optimal winning probability is $17/24$, an improvement of $6/24$ when compared with the optimal winning probability (equal to $11/24$) for the case when only one selection is allowed (see Figure~\ref{tree-1}). Note that for each permutation $\sigma \in S_4$, we list the probabilities $Q_1, Q_1^o$ in the first line and the probabilities $Q_2, Q_2^o$ in the second line underneath each permutation shown in Figure~\ref{tree-1}.
\end{example}

\begin{figure}
\begin{center}
  \includegraphics[scale=0.65]{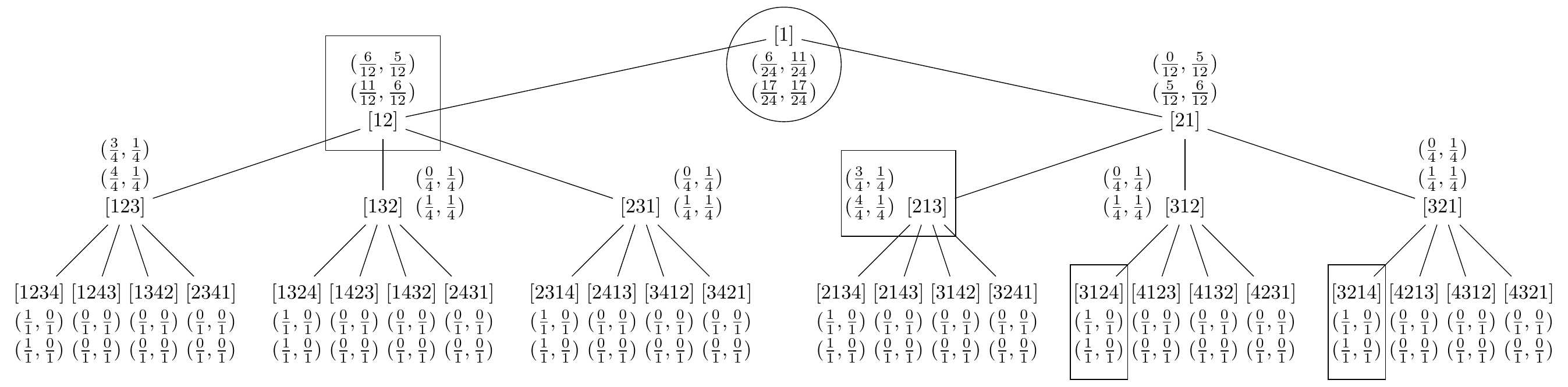}
  \caption{The prefix tree, $Q_1, Q^o_1, Q_2, Q^o_2$ probabilities , and a $2$-strike set for our problem with 4 applicants.}\label{tree-1}
\vspace{-8mm}
\end{center}
\end{figure}


\section{The optimal strategy}\label{optimal}

The maximum probability of winning for the Dowry model with $s$ selections and the query-based model with $s-1$ queries are the same, as both models have a budget of $s$ selections and the goal is to choose the best applicant. However, the expected stopping times are very different. Under the query-based model, the process immediately terminates after obtaining a positive answer from the expert. On the other hand, the decision making entity continues to interview the remaining applicants after a selection is made (if there is a selection left) under the Dowry setting, as it has no information whether the current applicant is the best. 

To obtain the optimal strategy for the Dowry model, we observe that $Q^o_i(\sigma) \ge Q_{i-1}^o(\sigma)$, $Q_i(\sigma) \ge Q_{i-1}(\sigma)$, and $Q_i(\sigma) \ge Q_{i-1}^o(\sigma)$ hold true for every $\sigma \in \bigcup\limits_{k=1}^{N}$ and $1 \le i \le s-1$. Let $Q_i(k)$ denote the probability $Q_i$ of eligible prefixes of length $k$, where $1 \le i \le s$. The probabilities $Q^o_i(k)$ and $\bar{Q}_i(k)$, where $1 \le i \le s$, are defined similarly.

\begin{theorem}\label{strategy}
An optimal strategy for our problem with $s$ selections is a positional $s$-thresholds strategy, i.e., there are $s$ numbers $0 \le k_1 \le k_2 \le \ldots \le k_s \le N$ such that when considering the $i^{\text{th}}$ selection, where $1 \le i \le s$, we reject the first $k_{i}$ applicants, wait for the $(i-1)^{\text{th}}$ selection,  and then accept the next left-to-right maximum. 
\end{theorem}

\begin{proof}
For the $i$-th selection ($j = s+1-i$ selections left), $1 \le i \le s$, by the algorithm described in Lemma~\ref{expansion}, we check if $\sigma$ is eligible and type $j$-positive (i.e., if $Q_j(\sigma) \ge Q^o_j(\sigma)$); if yes, we accept the current applicant and continue to the next selection (if one is left); if no, we reject the current applicant and continue our search; if there are no selections left, we terminate the process. By Corollary~\ref{cor-2}, we know each $Q_j(k) - Q_j^o(k)$ is a non-decreasing function of $k$, which allows us to formulate the optimal strategy.

By Corollary~\ref{cor-2}, and since we know $Q_j(N) \ge Q_j^o(N)$, there exists some $0 \le k_i \le N$ such that $Q_{j}(k) \ge Q^o_{j}(k)$ for $k \ge k_i+1$ and $Q_{j}(k) < Q^o_{j}(k)$ for all $k \le k_i$, where $1 \le i \le s$. Therefore, an optimal strategy is to reject the first $k_i$ applicants and then accept the next left-to-right maximum thereafter. It is also clear that every optimal strategy needs to proceed until the $(i-1)^{\text{th}}$ selection is made before considering the $i^{\text{th}}$ selection. Thus, $k_{i-1} \le k_i$ for each $i \in \{2, \ldots, s\}$.
\end{proof}

By the definition of the $Q_j(k), Q^o_j(k), \bar{Q}_j(k)$ probabilities, we know that they only depend on $k, N$ and the number of selections left before interviewing the current applicant, i.e., the subscript $j$. Thus, for two different models where the number of selections are $s_1$ and $s_2$ respectively (say $s_1 < s_2$), with the same value of $N$, each of the thresholds $k'_{s_1+1 - j}$ for the model with $s_1$ selections and $k''_{s_2+1 - j}$ for the model with $s_2$ selections are the same, where $1 \le j \le s_1$. In other words, our optimal strategy is right-hand based; and Corollary~\ref{fixed} holds.

\begin{corollary}\label{fixed}
Let $N$ be a fixed positive integer. There is a sequence of numbers $a_1, a_2, \ldots,$ such that when the number of selections $s \ge 1$ is fixed, then an optimal strategy is the $(a_s, a_{s-1}, \ldots, a_1)$-strategy. In other words, the $(s+1-i)^{\text{th}}$ threshold $k_{s+1-i}$ (the $i^{\text{th}}$ from the right) does not depend on the total number of selections allowed, i.e. the value of $s$, and always equals $a_i$, for $1\le i \le s$.
\end{corollary}

\vspace{-3mm}
\section{Solving the problem}\label{solve}
From previous section, we know that an optimal strategy is a positional $s$-thresholds strategy.

\begin{defn}
 Let $0 \le k_1 \le k_2 \le \ldots \le k_s \le N$. The $(k_1, k_2, \ldots, k_s)$-strategy is defined to be the strategy that when we are making the $i^{\text{th}}$ selection, where $1 \le i \le s$, we (1) wailt until the $(i-1)^{\text{th}}$ selection to be made (if $i=1$ then there is no need to wait) and (2) reject the first $k_i$ applicants and then accept the next left-to-right maximum.
\end{defn}

\begin{defn}
Let $0 \le r \le s$. We call a permutation $\pi$, which has length at most $N$, a  {\em $(k_1, k_2, \ldots, k_s)$-$r$-choosable permutation} if we can make at most $r$ selections when applying the $(k_1, k_2, \ldots, k_s)$-strategy on $\pi$. If the $(k_1, k_2, \ldots, k_s)$ strategy is fixed in the context, then we use $r$-choosable permutation in short.
\end{defn}

\begin{defn}
Let $0 \le k_1 \le k_2 \le \ldots \le k_s \le N$. We denote the number of $(k_1, k_2, \ldots, k_s)$-$r$-choosable permutations in $S_N$ by $T_r(N, k_1, k_2, \ldots, k_s)$. Furthermore, we denote the number of permutations in $S_N$ such that the value $N$ can be selected using the $(k_1, k_2, \ldots, k_s)$-strategy by $W(N, k_1, k_2, \ldots, k_s)$.
\end{defn}

Note that the winning probability by using the $(k_1, k_2, \ldots, k_s)$-strategy is $W(N, k_1,k_2, \ldots, k_s)/N!.$

\vspace{-2mm}
\subsection{The recurrence relations}
Note that if $s \ge r+1$ then $k_r$ is the only threshold that matters when we are calculating the number of permutations which result in at most $r-1$ selections. Therefore, we focus on $T_{r-1}(m, k_1, \ldots, k_r)$.

\begin{lemma}\label{rec-1}
Let $0 \le k_1 \le k_2 \le \ldots \le k_s \le N$ be fixed. Let $1 \le r \le s$ and $m \ge k_r$. We claim that the following recurrence relation holds:
\begin{equation}
T_{r-1}(m, k_1, \ldots, k_r) = k_r \cdot (m-1)! + (m-1)! \cdot \sum\limits_{i=k_r+1}^{m} \frac{T_{r-2}(i-1, k_1, \ldots, k_{r-1})}{(i-1)!}.
\end{equation}
\end{lemma}

\begin{proof}
There are two cases.


\textbf{Case 1:} The value $m$ is at a position $i \in [1, k_r]$, then (1) at most $r-1$ selections can be made before position $i+1$ since $k_r+1$ is the first position that we can start making the $r^{\text{th}}$ selection and (2) no further selections can be made after position $i$ since there is no left-to-right maximum after we have seen the best applicant (value $m$) at position $i$. This case contributes $k_r \cdot (N-1)!$.


\textbf{Case 2:} The value $m$ is at a position $i \in [k_r+1, m]$. Then we require at most $r-2$ selections can be made from positions $[1, i-1]$, since otherwise either $r$ selections was already made before position $i$ or we make the $r^{\text{th}}$ selection at position $i$. For each case when the value $m$ is at position $i \in [k_r+1, m]$, we first choose $i-1$ values from $\{1, \ldots, m-1\}$ to be placed at the first $i-1$ positions, the first $i-1$ positions (with its relative ordering) form a $(r-2)$-choosable permutation of length $i-1$ using the $(k_1, \ldots, k_r)$-strategy, and there is no restriction on the last $m-i$ positions. Therefore, each $i \in [k_r+1, m]$ contributes $${m-1 \choose i-1} \cdot T_{r-2}(i-1, k_1, \ldots, k_r) \cdot (m-i)! \quad \text{i.e.} \quad \frac{(m-1)!}{(i-1)!} \cdot T_{r-2}(i-1, k_1, \ldots, k_{r-1}). \qedhere$$
\end{proof}

\begin{lemma}\label{rec-2}
Let $0 \le k_1 \le k_2 \le \ldots \le k_s \le N$ be fixed. Let $m \ge k_{r}+1$ and $1 \le r \le s$. We claim that the following recurrence relation holds:
\begin{equation}
W(m, k_1, \ldots, k_r) = T_{r-1}(m-1, k_1, \ldots, k_r) + (m-1) \cdot W(m-1, k_1, \ldots, k_r).
\end{equation}
\end{lemma}

\begin{proof}
There are two cases depending on the value on the last position.

\textbf{Case 1:} The last position is $m$. In this case, we require that at most $r-1$ selections was made from position $1$ to $m-1$, since otherwise we do not have enough selections left to choose the best applicant at the last position. This case contributes $T_{r-1}(m-1, k_1, \ldots, k_r)$.

\textbf{Case 2:} The last position is a value from $\{1,2, \ldots, m-1\}$. Then we must select the best applicant (value $m$) from position $1$ to $m-1$. This case contributes $ (m-1) \cdot W(m-1, k_1, \ldots, k_r)$.
\end{proof}

\subsection{Solving the recurrence relations}
For each $1 \le r \le s$, we know that $T_{r-1}(m, k_1, \ldots, k_r) = m!$ for each $m \le k_r$, since we can make at most $r-1$ selections by the definition of the $(k_1, \ldots, k_r)$-strategy. Furthermore, for $r = 1$ and $m \ge k_1$, we know $T_0(m, k_1) = k_1 \cdot (m-1)!$, since it corresponds to the cases that the best applicant is located on $[1, k_r]$.

\begin{lemma}\label{sol-1}
Let $1 \le r \le s$ and $m \ge k_r + 1$. We have 
$$\frac{T_{r-1}(m, k_1, \ldots, k_r)}{m!} = \frac{1}{m} \cdot (k_r + k_{r-1} \cdot \sum\limits_{i = k_r}^{m-1} \frac{1}{i}+ k_{r-2} \cdot \sum\limits_{i_1 = k_r}^{m-1} \frac{1}{i_1} \sum\limits_{i_2 = k_{r-1}}^{i_1-1} \frac{1}{i_2}  + \ldots +$$
$$k_1 \cdot \sum\limits_{i_1 = k_r}^{m-1} \frac{1}{i_1} \cdot \sum\limits_{i_2 = k_{r-1}}^{i_1-1} \frac{1}{i_2} \sum \cdots \sum\limits_{i_{r-1} = k_2}^{i_{r-2}-1} \frac{1}{i_{r-1}}).$$
\end{lemma}

\begin{remark}
In our notation, if the upper limit is less than the lower limit in a summation, then the corresponding term in the summation does not exist. For example, if $k_r = k_{r-1}$ then $k_r - 1 < k_{r-1}$ and the inner summation $\sum\limits_{i_2 = k_{r-1}}^{i_1-1} \frac{1}{i_2}$ will skip the case when $i_1 = k_r$ and will start with $i_1 = k_r+1$ in $\sum\limits_{i_1 = k_r}^{m-1} \frac{1}{i_1} \sum\limits_{i_2 = k_{r-1}}^{i_1-1} \frac{1}{i_2}$.
\end{remark}

\begin{proof}
The proof follows by induction. We already knew the case when $r = 1$ holds. Assume the result is true for all number of selections less than $r$, where $2 \le r \le s$. We show that it holds for $r$ as well. By Lemma~\ref{rec-1}, $T_{r-1}(m,k_1, \ldots, k_r)$ can be expressed in terms of $T_{r-2}(i, k_1, \ldots, k_{r-1})$, where $k_r \le i \le m-1$. Therefore, by using the formula for $T_{r-2}(i, k_1, \ldots, k_{r-1})$, guaranteed by the inductive hypothesis, we obtain the claimed formula for $\frac{T_{r-1}(m, k_1, \ldots, k_r)}{m!}$. The lengthy computations are omitted.
\end{proof}

\begin{theorem}
Let $s \ge 1$ be fixed and $N \ge k_s+1$. Then 
$$\frac{W(N, k_1, \ldots, k_s)}{N!} = \frac{1}{N} \cdot \Bigg( \Big(k_1 \cdot \sum\limits_{i = k_1}^{k_2-1} \frac{1}{i}  \Big)  +  \Big(k_2 \cdot \sum\limits_{i = k_2}^{k_3-1} \frac{1}{i} +  k_1 \cdot \sum\limits_{i_1 = k_2+1}^{k_3-1} \frac{1}{i_1} \sum\limits_{i_2=k_2}^{i_1-1} \frac{1}{i_2}  \Big) + $$
$$\Big(k_3 \cdot \sum\limits_{i=k_3}^{k_4-1} \frac{1}{i} + k_2 \cdot \sum\limits_{i_1 = k_3+1}^{k_4-1} \frac{1}{i_1} \sum\limits_{i_2 = k_3}^{i_1-1} \frac{1}{i_2} + k_1 \cdot \sum\limits_{i_1 = k_3+1}^{k_4-1} \frac{1}{i_1} \sum\limits_{i_2 = k_3}^{i_1-1} \frac{1}{i_2} \sum\limits_{i_3 = k_2}^{i_2-1} \frac{1}{i_3}  \Big)
$$
$$ + \ldots +  $$
\begin{equation}\label{winningprob}
 \Big( k_s \cdot \sum\limits_{i = k_s}^{N-1} \frac{1}{i} + k_{s-1} \cdot \sum\limits_{i_1 = k_s+1}^{N-1} \frac{1}{i_1} \sum\limits_{i_2 = k_s}^{i_1-1} \frac{1}{i_2} + \ldots + k_1 \cdot \sum\limits_{i_1 = k_s+1}^{N-1} \frac{1}{i_1} \sum\limits_{i_2 = k_s}^{i_1-1} \frac{1}{i_2} \cdots \sum\limits_{i_s = k_2}^{i_{s-1}-1} \frac{1}{i_s} \Big) \Bigg).
\end{equation}
\end{theorem}

\begin{proof}
We prove it by induction. The case when $s = 1$ holds true since it is a corollary of a result of Jones~\cite{jones2020weighted}. Assume the claim holds for $W(N, k_1, \ldots, k_r)/N!,$ where $r < s$. We prove the result holds for $r = s$. 

By Lemma~\ref{rec-2}, $$W(N, k_1, \ldots, k_s) = (N-1) \cdot W(N-1, k_1, \ldots, k_s) + T_{s-1}(N-1, k_1, \ldots, k_s)$$
$$= (N-1) \cdot ((N-2) \cdot W(N-2, k_1, \ldots, k_s) + T_{s-1}(N-2, k_1, \ldots, k_s)) + T_{s-1}(N-1, k_1, \ldots, k_s) = \ldots = $$
$$\frac{(N-1)!}{(k_s-1)!}\cdot W(k_s, k_1, \ldots, k_{s-1}, k_s) + \sum\limits_{i = 1}^{N-k_s} \frac{(N-1)!}{(N-i)!} \cdot T_{s-1}(N-i, k_1, \ldots, k_s)=$$
$$\frac{(N-1)!}{(k_s-1)!}\cdot W(k_s, k_1, \ldots, k_{s-1}) + \sum\limits_{i = 1}^{N-k_s} \frac{(N-1)!}{(N-i)!} \cdot T_{s-1}(N-i, k_1, \ldots, k_s).$$

By the inductive hypothesis, we can use the formula for $W(k_s, k_1, \ldots, k_{s-1})$, and the formula of \\$T_{s-1}(N-i, k_1, \ldots, k_s)$ in Lemma~\ref{rec-1} to establish the claim. The lengthy computations are omitted.
\end{proof}

By Corollary~\ref{fixed}, we know that an optimal strategy is right hand based. Let $s \ge 1$ be fixed. We focus on $$\frac{W(N, a_s, a_{s-1}, \ldots, a_1)}{N!}.$$ 
Let $$H_1 =  \frac{a_1}{N} \cdot \sum\limits_{i = a_1}^{N-1} \frac{1}{i} ,\quad H_2 = \frac{a_2}{N} \cdot (  \sum\limits_{i = a_2}^{a_1-1} \frac{1}{i} + \sum\limits_{i_1 = a_1+1}^{N-1} \frac{1}{i_1} \sum\limits_{i_2=a_1}^{i_1-1} \frac{1}{i_2}),$$
$$H_3 = \frac{a_3}{N} \cdot (\sum\limits_{i=a_3}^{a_2-1} \frac{1}{i} +  \sum\limits_{i_1 = a_2+1}^{a_1-1} \frac{1}{i_1} \sum\limits_{i_2 = a_2}^{i_1-1} \frac{1}{i_2} +  \sum\limits_{i_1 = a_1+1}^{N-1} \frac{1}{i_1} \sum\limits_{i_2 = a_1}^{i_1-1} \frac{1}{i_2} \sum\limits_{i_3 = a_2}^{i_2-1} \frac{1}{i_3})$$
$$\ldots $$
$$H_s = \frac{a_s}{N} \cdot (\sum\limits_{i = a_s}^{a_{s-1}-1} \frac{1}{i} +  \sum\limits_{i_1 = a_{s-1}+1}^{a_{s-2}-1} \frac{1}{i_1} \sum\limits_{i_2 = a_{s-1}}^{i_1-1} \frac{1}{i_2} + \sum\limits_{i_1 = a_{s-2}+1}^{a_{s-3}-1} \frac{1}{i_1} \sum\limits_{i_2 = a_{s-2}}^{i_1-1} \frac{1}{i_2} \sum\limits_{i_3 = a_{s-1}}^{i_2-1} \frac{1}{i_3} + \ldots + $$
$$\sum\limits_{i_1 = a_1+1}^{N-1} \frac{1}{i_1} \sum\limits_{i_2 = a_1}^{i_1-1} \frac{1}{i_2} \sum\limits_{i_3 = a_2}^{i_2-1} \frac{1}{i_3}\cdots \sum\limits_{i_s = a_{s-1}}^{i_{s-1}-1} \frac{1}{i_s})).$$ 

By rearranging equation~\eqref{winningprob}, we obtain that $$\frac{W(N, a_s, \ldots, a_1)}{N!} = \sum\limits_{i=1}^{s} H_i.$$

\vspace{-8mm}
\subsection{the asymptotic case ($N \to \infty$)}
Let $N \to \infty$ and define $x_i = \lim\limits_{N \to \infty} \frac{a_i}{N}$. For $0 \le x_s \le x_{s-1} \le \ldots \le x_1 \le 1$, and $1 \le r \le s$,
$$H_r \to x_r \cdot (\int\limits_{x_r}^{x_{r-1}} \frac{1}{t} \,dt +  \int\limits_{x_{r-1}}^{x_{r-2}} \frac{1}{t_1} \int\limits_{x_{r-1}}^{t_1} \frac{1}{t_2} \,dt_2\,dt_1 + \int\limits_{x_{r-2}}^{x_{r-3}} \frac{1}{t_1} \int\limits_{x_{r-2}}^{t_1} \frac{1}{t_2} \int\limits_{x_{r-1}}^{t_2} \frac{1}{t_3} \,dt_3\,dt_2\,dt_1  + \ldots +$$
$$ \int\limits_{x_1}^{1} \frac{1}{t_1} \int\limits_{x_1}^{t_1} \frac{1}{t_2} \int\limits_{x_{2}}^{t_2} \cdots \int\limits_{x_{r-1}}^{t_{r-1}} \frac{1}{t_r} \,dt_r\,dt_{r-1} \cdots \,dt_1) =: H_r'$$
$$\lim\limits_{N \to \infty}\frac{W(N, a_s, \ldots, a_1)}{N!} = \sum\limits_{r = 1}^{s} H_r' =: P.$$

By this definition, and for $1 \le i \le s$, the optimal probability is a function of $x_1, \ldots, x_s$ and does not depend on $N$. By Corollary~\ref{fixed}, we find the ratio $x_1$ for the case $s = 1$, then we find $x_2$ for the case $s=2$, $\ldots$, finally we find $x_s$ for the case of $s$ selections allowed. An optimal strategy has to satisfy that if: (i) $a_i = o(N)$, then $x_i = 0$; (ii) $N-a_i = o(N)$, then $x_i = 1$; (iii) and, otherwise, $\lim\limits_{N \to \infty} \frac{a_i}{N} = x_i \in (0,1)$. Next, we show that only (iii) is possible, i.e., each $x_i \in (0,1)$, for $1 \le i \le s$.

Let $$\mathcal{I}_{r-1} = \int\limits_{x_{r-1}}^{x_{r-2}} \frac{1}{t_1} \int\limits_{x_{r-1}}^{t_1} \frac{1}{t_2} \,dt_2\,dt_1 + \int\limits_{x_{r-2}}^{x_{r-3}} \frac{1}{t_1} \int\limits_{x_{r-2}}^{t_1} \frac{1}{t_2} \int\limits_{x_{r-1}}^{t_2} \frac{1}{t_3} \,dt_3\,dt_2\,dt_1  + \ldots + \int\limits_{x_1}^{1} \frac{1}{t_1} \int\limits_{x_1}^{t_1} \frac{1}{t_2} \int\limits_{x_{2}}^{t_2} \cdots \int\limits_{x_{r-1}}^{t_{r-1}} \frac{1}{t_r} \,dt_r \,dt_{r-1} \cdots \,dt_1.$$ 

We know that $\mathcal{I}_{r-1}$ is a constant when $x_1, \ldots, x_{r-1}$ are given and the optimal value of $x_r$ which realizes the maximum value of $P$ equals $x_{r-1} \cdot e^{\mathcal{I}_{r-1}-1}$. Thus, we can compute $x_{r-1}$ and $\mathcal{I}_{r-1}$ sequentially.

We computed the optimal thresholds for the case of up to five selections and their corresponding  probabilities of success. When $N$ is large, by Corollary~\ref{fixed}, an optimal strategy for the case of $s$ selections is the $(x_s \cdot N, \ldots, x_1 \cdot N)$-strategy: We reject the first $x_s \cdot N$ applicants and then select the first left-to-right maximum applicant thereafter; for the second selection, we wait until after the $x_{s-1} \cdot N^{\text{th}}$ position and then select the left-to-right maximum after the first selection made, $\ldots$, for the $s^{\text{th}}$ selection, we wait until after the $x_1 \cdot N^{\text{th}}$ position and then select the left-to-right maximum after the $(s-1)^{\text{th}}$ selection is made. The results for $s \leq 5$ are shown in Table~\ref{uniform}.

\begin{center}
\begin{table}
\begin{tabular}{|c|c|c|c|c|c|c|c|c}
\hline
 & $x_1$ & $x_2$ & $x_3$ & $x_4$ & $x_5$ & $\ldots$  \\ \hline

Thresholds & 0.3678794412& 0.2231301601 & 0.1410933807 & 0.0910176906   &  0.0594292419  & \ldots \\ \hline
 
Optimal probability & 0.3678794412&0.5910096013&0.7321029820& 0.8231206726&0.8825499146 & \ldots \\ \hline
\end{tabular}
\caption{The optimal thresholds (threshold ratios) and success probabilities.}\label{uniform}
\vspace{-5mm}
\end{table}
\end{center}

\vspace{-15mm}

\subsection{The expected stopping position}

The expected position of stopping divided by $N$ (ESR in short) for both the Dowry model and our query-base model are $0.7357 N$ when $s=1$. However, when $s$ is large, the two models behave very differently. In the Dowry model, ESR approaches $1$ as $s$ becomes large. In the query-based model, ESR approaches $0.5$ as $s \to \infty$. This follows since when $s$ is sufficiently large, we have a success probability close to $1$ and will stop at the position at which the value $N$ appears if this is identified during the $s-1$ queries (even for $s = 6$ we will have a probability $> 0.9$ of identifying the best applicant). Thus, ESR $\to 0.5,$ the expected position of $N$. However, in the Dowry model, the probability of successfully identifying the best applicant during the first $s-1$ selections $\to 1$ as $s$ increases. Furthermore, since we have no information about whether the selected applicant is the best or not, we try our best to use all selections and thus have probabilities $\to 1$ of stopping at the end of the list. Thus, $ESR \to 1$ as $N \to \infty$ in the Dowry model.

\vspace{3mm}
\noindent \textbf{Acknowledgment.} The work was supported in part by the NSF CCF 15-26875 and CIF\,1513373 grants.


\begin{thebibliography}{99}\setlength{\itemsep}{-0.001mm}
\bibitem{AKB1}
H. Ashtiani, S. Kushagra, S. Ben-David, ``Clustering with same-cluster queries,'' \emph{Advances in Neural Information Processing Systems (NIPS)}, pp. 3224--3232, 2016.

\bibitem{BIK1}
M. Babaioff, N. Immorlica, R. Kleinberg,  ``Matroids, secretary problems, and online mechanisms'' {\em ACM-SIAM Symposium on Discrete Algorithms (SODA)}, pp. 434--443, 2007.


\bibitem{BFGOR1}
L. Bay\'on, P. Fortuny Ayuso, J.M. Grau, A.M. Oller-Marc\'en and M. M. Ruiz, ``The Best-or-Worst and the Postdoc problems'' {\em J Comb Optim}, vol. 35, pp. 703--723, 2018. 

\bibitem{BFGOR2}
L. Bay\'on, P. Fortuny Ayuso, J.M. Grau, A.M. Oller-Marc\'en and M. M. Ruiz, ``The Best-or-Worst and the Postdoc problems with random number of candidates'' {\em J Comb Optim}, vol. 38, pp. 86--110, 2019.

\bibitem{CPM1}
I. Chien, C. Pan, and O. Milenkovic, ``Query k-means clustering and the double dixie cup problem,'' \emph{Advances in Neural Information Processing Systems (NeurIPS)}, 31, pp. 6649-6658, 2018.


\bibitem{CJMTUW1}
M. Crews, B. Jones, K. Myers, L. Taalman, M. Urbanski and B. Wilson, ``Opportunity Costs in the Game of Best Choice, '' {\em The Electronic Journal of Combinatorics}, vol. 26, no. 1, \#P1.45, 2019.

\bibitem{D1}
E.B. Dynkin, ``The optimal choice of the stopping moment for a Markov process,'' {\em Dokl. Akad. Nauk. SSSR}, vol. 150, pp. 238--240, 1963.

 

\bibitem{FJ19}
A. Fowlkes and B. Jones, Positional strategies in games of best choice, {\em Involve}, vol. 12, no. 4, 647--658, 2019.
MR 3941603
 
\bibitem{F2}
P. R. Freeman, ``The secretary problem and its extensions - A review,'' {\em Internat. Statist. Rev.} vol.51, pp. 189--206, 1983.
 
\bibitem{FW1}
R. Freij, J. Wastlund, ``Partially ordered secretaries,'' {\em Electron Commun Probab}, vol. 15, pp. 504--507, 2010.

\bibitem{G1}
M. Gardner, ``Mathematical games'', {\em Scientific American}, vol. 202, no. 2, pp. 152, 1960a.

\bibitem{G2}
M. Gardner, ``Mathematical games'', {\em Scientific American}, vol. 202, no. 3, pp. 178--179, 1960b.

\bibitem{GM1}
B. Garrod and R. Morris, ``The secretary problem on an unknown poset,'' {\em Random Structures \& Algorithms}, vol. 43, pp. 429--451, 2012.

\bibitem{GKMN1}
N. Georgiou, M. Kuchta, M. Morayne, J. Niemiec, ``On a universal best choice algorithm for partially ordered sets'', {\em Random Structures \& Algorithms}, vol. 32, pp. 263--273, 2008.
 
\bibitem{GM2}
J. Gilbert and F. Mosteller, ``Recognizing the maximum of a sequence,'' {\em J. Amer. Statist. Assoc.}, vol. 61, pp. 35--73, 1966.




\bibitem{jones2019}
B. Jones, ``Avoiding patterns and making the best choice,'' {\em Discrete Mathematics}, vol. 342, no. 6, pp.1529-1545, 2019.

\bibitem{jones2020weighted}
B. Jones, ``Weighted games of best choice,'' {\em SIAM Journal on Discrete Mathematics}, vol. 34, no. 1, pp. 399--414, 2020.

\bibitem{KS}
U. Krengel and L. Sucheston, ``On semiamarts, amarts, and processes with finite value'', {\em Probability on Banach spaces, Adv. Probab. Related Topics}, vol. 4, pp. 197--266. Dekker, New York, 1978.

\bibitem{L1}
D.V. Lindley, ``Dynamic programming and decision theory,'' {\em Appl. Statist.}, vol.10, pp. 39--52, 1961. 

\bibitem{LM1} 
X. Liu and O. Milenkovic, ``The Postdoc Problem under the Mallows Model,'' 2021 IEEE International Symposium on Information Theory (ISIT), 2021, pp. 3214--3219.

\bibitem{LM2}
X. Liu and O. Milenkovic, ``Finding the second-best candidate under the Mallows model'', {\em Theoretical Computer Science}, vol. 929, 39-68, 2022. 

\bibitem{MS1}
A. Mazumdar and B. Saha, ``Clustering with noisy queries,'' {\em Advances in Neural Information Processing Systems (NIPS)}, pp. 5788-5799, 2017.

\bibitem{MLM1}
G. Moustakides, X. Liu, and O. Milenkovic, ``Optimal Stopping Methodology
for the Secretary Problem with Random Queries,'' https://arxiv.org/pdf/2107.07513.pdf.


































\bibitem{R1}
J. S. Rose, ``A problem of optimal choice and assignment,'' {\em Operations Research}, vol. 30, pp. 172--181, 1982.

\bibitem{S1}
M. Sakaguchi, ``Dowry problems and OLA policies,'' {\em Rep. Stat. Appl. Res., Juse}, vol. 25, 124--128, 1978.

\bibitem{soto}
J.A. Soto, ``Matroid secretary problem in the random assignment model'', {\em ACM-SIAM Symposium on Discrete Algorithms (SODA)}, pp. 1275--1284, 2011.


 













\end{thebibliography}
\end{document}